\documentclass{birkau}

\usepackage{amsmath,amssymb,url}
\usepackage{graphicx}
\usepackage{bussproofs}
\usepackage[all, 2cell]{xy}
\usepackage[all]{xy}

\numberwithin{equation}{section}

\theoremstyle{plain}
\newtheorem{theorem}{Theorem}[section]
\newtheorem{lemma}[theorem]{Lemma}
\newtheorem{proposition}[theorem]{Proposition}
\newtheorem{corollary}[theorem]{Corollary}

\theoremstyle{definition}
\newtheorem{definition}[theorem]{Definition}

\usepackage{multirow}

\begin{document}


\title[The finite representation property]{The finite representation property for some reducts of relation algebras}



\author[D. Rogozin]{Daniel Rogozin}
\address{Institute for Information Transmission Problems, Russian Academy of Sciences}
\email{daniel.rogozin@serokell.io}

\thanks{The author is sincerely grateful to Robin Hirsch, Ian Hodkinson, Stepan Kuznetsov, Valentin Shehtman, Eugeny Zolin, and his supervisor Ilya Shapirovsky for valuable conversations.\\The previous version contained an error. Thanks to Jas Semrl for noticing that.}


\subjclass{03G27, 03G15, 06B15}

\keywords{Relation algebras, The finite representation property, Residuated semigroups, Upper semilattice-ordered semigroups}

\begin{abstract}
In this paper, we show that the class of representable residuated semigroups has the finite representation property. That is, every finite representable residuated semigroup is isomorphic to some algebra over a finite base. This result gives a positive solution to Problem 19.17 from the monograph by Hirsch and Hodkinson \cite{hirsch2002relation}.
\end{abstract}

\maketitle

\section{Introduction}

Relation algebras are the kind of Boolean algebras with operators representing algebras of binary relations \cite{jonsson1951boolean}. One often emphasise the following two classes of relation algebras.
The first class called ${\bf RA}$ consists of algebras the signature of which is $\{ 0, 1, +, -, ;, {}^{\smile}, {\bf 1}' \}$ obeying the certain axioms that we define precisely below.
The second class called ${\bf RRA}$, the class of representable relation algebras, consists of algebras isomorphic to set relation algebras. ${\bf RRA}$ is a subclass of ${\bf RA}$, but the converse inclusion does not hold. That is, there exist relation algebras having no representation as set relation algebras \cite{lyndon1950representation}. Moreover, the class ${\bf RRA}$ is not finitely axiomatisable in contrast to ${\bf RA}$ \cite{monk1964representable}. The problem of determining whether a given finite relation algebra is representable is undecidable, see \cite{hirsch2001representability}.

Under these circumstances, one may conclude that relation algebras are quite badly behaved. The study of such reducts is mostly motivated by such ``bad behaviour'' in order to avoid these restrictions and determine the possible reasons for them.

There are several results on reducts of relation algebras having no finite axiomatisation such as ordered monoids \cite{hirsch2005class}, distributive residuated lattices \cite{andreka1994lambek}, join semilattice-ordered semigroups \cite{andreka2011axiomatizability}, algebras whose signature contains composition, meet, and converse \cite{hodkinson2000axiomatizability}, etc.

On the other hand, such classes as representable residuated semigroups and monoids \cite{andreka1994lambek}, and ordered domain algebras \cite{hirsch2013ordered} are finitely axiomatisable. There are also plenty of subsignatures for which the question of finite axiomatisability remains open, see, e. g., \cite{andreka2011axiomatizability}.

The other direction we emphasise is related to finite representability. A finite algebra of relations has the finite representation property if it is isomorphic to some algebra of relations on a finite base. The investigation is of interest to study such aspects as the decidability of membership of ${\bf R}(\tau)$. The finite representation property also implies recursivity of the class of all finite representable $\tau$-structures \cite{hirsch2004finite}, if the whole class is finitely axiomatisable.  Here, $\tau$ is a subsignature of operations definable in $\{ 0, 1, +, -, ;, {}^{\smile}, {\bf 1 } \}$. The examples of the class having the finite representation property are some classes of algebras \cite{hirsch2004finite} \cite{hirsch2013ordered} \cite{mclean2016finite}, the subsignature of which contains the domain and range operators. The other kind of algebras of binary relations having the finite representation property is semigroups with so-called demonic refinement has been recently studied by Hirsch and \v{S}emrl \cite{hirsch2020finite}.

Other algebras of relations such as weakly associative or relativised cylindric set algebras of a finite dimension have a similar property called the finite base property \cite{andreka1999finite}.

There are subsignatures of relation algebras $\tau$ such that
the class ${\bf R}(\tau)$ of representable reducts fails to have the finite representation property. For instance, $\{;, \cdot\}$, see \cite{maddux2016finite}. In general, (un)decidability of determining whether a finite relation algebra has a finite representation is an open question \cite[Problem 18.18]{hirsch2002relation}.

In this paper, we consider reducts of relation algebras the signature of which consists of composition, residuals, and the binary relation symbol denoting partial ordering, that is, the class of representable residuated semigroups. We show that ${\bf R}(;, \setminus, /, \leq)$ has the finite representation property. As result, Problem 19.17 of \cite{hirsch2002relation} has a positive solution. We also note that this result implies of membership decidability of ${\bf R}(;, \setminus, /, \leq)$ for finite structures and the class of finite representable residuated semigroups is recursive. The solution is surprisingly simple and based on the Dedekind-MacNeille completion and the relation representation of quantales. We consider a finite residuated semigroup and embed into a finite quantale mapping every element to its lower cone. After that, we apply the relational representation for quantales. As a result, the original finite residuated semigroup has a Zaretski-style representation \cite{zaretskii1959representation} and this satisfies the finite base requirement.

Finally, we show that  ${\bf R}(;, +)$, the class of representable upper semilattice-ordered semigroups, has no finite representation property. We provide a $\{+, ;\}$-reduct of the point algebra the base of which is always infinite.

\section{Preliminaries}

Let us the basic definitions related to relation algebras \cite{hirsch2002relation}.
\begin{definition} A relation algebra is an algebra $\mathcal{R} = \langle R, 0, 1, +, -, ;, {}^{\smile}, {\bf 1 }\rangle$ such that $\langle R, 0, 1, +, - \rangle$ is a Boolean algebra and the following
    equations hold, for each $a, b, c \in R$:
    \begin{enumerate}
      \item $a ; (b ; c) = (a ; b) ; c$,
      \item $(a + b) ; c = (a ; c) + (b ; c)$,
      \item $a ; {\bf 1}' = a$,
      \item $a^{\smile \smile} = a$,
      \item $(a + b)^{\smile} = a^{\smile} + b^{\smile}$,
      \item $(a ; b)^{\smile} = b^{\smile} ; a^{\smile}$,
      \item $a^{\smile} ; (- (a ; b)) \leq - b$.
    \end{enumerate}
    Note that $a \leq b$ iff $a + b = b$ iff $a \cdot b = a$, where $a \cdot b = - (- a + - b)$. ${\bf RA}$ is the class of all relation algebras.
\end{definition}

\begin{definition}
    A proper relation algebra (or, a set relation algebra) is an algebra $\mathcal{R} = \langle R, 0, 1, \cup, -, ;, {}^{\smile}, {\bf 1 }\rangle$ such that $R \subseteq \mathcal{P}(W)$, where $X$ is a base set, $W \subseteq X \times X$ is an equivalence relation, $0 = \emptyset$, $1 = W$, $\cup$ and $-$ are set-theoretic union and complement respectively, $;$ is relation composition, ${}^{\smile}$ is relation converse,
    ${\bf 1}'$ is the identity relation restricted to $W$, that is:
    \begin{enumerate}
    \item $a ; b = \{ \langle x, z \rangle \in W \: | \: \exists y \: \langle x, y \rangle \in a \: \& \: \langle y, z \rangle \in b \}$
    \item $a^{\smile} = \{ \langle x, y \rangle \in W \: | \: \langle y, x \rangle \in a \}$
    \item ${\bf 1}' = \{ \langle x, y \rangle \in W \: | \: x = y \}$
    \end{enumerate}
       ${\bf PRA}$ is the class of all proper relation algebras. ${\bf RRA}$ is the class of all representable relation algebras, that is, the closure of ${\bf PRA}$ under isomorphic copies.
\end{definition}

We will use the following notation due to, for example, \cite{hirsch2011positive}. Let $\tau$ be a subset of operations and predicates definable in ${\bf RA}$. ${\bf R}(\tau)$ is the class of subalgebras of $\tau$-subreducts of algebras belonging to ${\bf RRA}$. We assume that ${\bf R}(\tau)$ is closed under isomorphic copies.

A $\tau$-structure is \emph{representable} if it is isomorphic to some algebra of relations of $\tau$-signature. A representable finite $\tau$-structure has a \emph{finite representation over a finite base} if it is isomoprhic to some finite representable over a finite base. ${\bf R}(\tau)$ has the finite representation property if every $\mathcal{A} \in {\bf R}(\tau)$ has a finite representation over a finite base.

One may express residuals in every $\mathcal{R} \in {\bf RA}$ as follows using Boolean negation, inversion, and composition as follows:

\begin{enumerate}
  \item $a \setminus b = -(a^{\smile} ; -b)$
  \item $a \: / \: b = - (- a ; b^{\smile})$
\end{enumerate}

These residuals have the following explicit definition in $\mathcal{R} \in {\bf PRA}$:
\begin{enumerate}
  \item $a \setminus b = \{ \langle x, y \rangle \: | \: \forall z \: (z, x) \in a \Rightarrow (z, y) \in b \}$
  \item $a \: / \: b = \{ \langle x, y \rangle \: | \: \forall z \: (y, z) \in b \Rightarrow (x, z) \in a \}$
\end{enumerate}

One may visualise residuals in ${\bf RRA}$ with the following triangles:

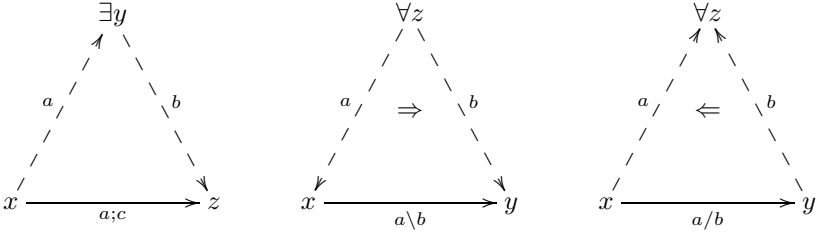
\xymatrix{
&& \exists y \ar@{-->}[ddr]^{b} &&& \forall z \ar@{-->}[ddl]_{a} \ar@{-->}[ddr]^{b} &&& \forall z \\
&&&&& \Rightarrow &&& \Leftarrow \\
& x \ar@{-->}[uur]^{a} \ar[rr]_{a;c} && z & x \ar[rr]_{a \setminus b} && y & x \ar@{-->}[uur]^{a} \ar[rr]_{a / b} && y \ar@{-->}[uul]_{b}
}

\section{The case of residuated semigroups}

The problem we are interested in is the following \cite[Problem 19.17]{hirsch2002relation}:

\begin{center}
  Does ${\bf R}(;, \setminus, /, \leq)$ have the finite representation property?
\end{center}

Let us introduce the notion of a residuated semigroup. Historically, residuated structures were introduced by Krull to study ideals of rings \cite{krull1968idealtheorie}. Structures of this kind has been considered further within semantic aspects of substructural logics, see, for example, \cite{jipsen2002survey}.

\begin{definition}
  A residuated semigroup is an algebra $\mathcal{A} = \langle A, ;, \leq, \setminus, / \rangle$ such that $\langle A, ;, \leq, \rangle$ is a partially ordered residuated semigroup and $\setminus, /$ are binary operations satisfying the residuation property:

  \begin{center}
    $b \leq a \setminus c \Leftrightarrow a ; b \leq c \Leftrightarrow a \leq c \: / \: b$
  \end{center}
  ${\bf RS}$ is the class of all residuated semigroups.
\end{definition}

The logic of such structures is the Lambek calculus \cite{lambek1958mathematics} allowing one to characterise inference in categorial grammars, the equivalent version of context-free grammars \cite{pentus1993lambek}. One may define the Lambek calculus as the following (cut-free) sequent calculus:

\begin{definition}
  $ $

  \begin{prooftree}
  \AxiomC{$ $}
  \RightLabel{\scriptsize{ax}}
  \UnaryInfC{$\varphi \vdash \varphi$}
  \end{prooftree}

\begin{minipage}{0.5\textwidth}
  \begin{flushleft}
        \begin{prooftree}
      \AxiomC{$\Gamma \vdash \varphi$}
      \AxiomC{$\Delta, \psi, \Theta \vdash \theta$}
      \RightLabel{$\backslash \vdash $}
      \BinaryInfC{$\Delta, \Gamma, \varphi \backslash \psi, \Theta \vdash \theta$}
    \end{prooftree}

    \begin{prooftree}
      \AxiomC{$\Gamma \vdash \varphi$}
      \AxiomC{$\Delta, \psi, \Theta \vdash \theta$}
      \RightLabel{$/ \vdash $}
      \BinaryInfC{$\Delta, \psi \: / \: \varphi, \Gamma, \Theta \vdash \theta$}
    \end{prooftree}

    \begin{prooftree}
      \AxiomC{$\Gamma, \varphi, \psi, \Delta \vdash \theta$}
      \RightLabel{$\bullet \vdash $}
      \UnaryInfC{$\Gamma, \varphi \bullet \psi, \Delta \vdash \theta$}
    \end{prooftree}
\end{flushleft}
\end{minipage}
\begin{minipage}{0.5\textwidth}
  \begin{flushright}
       \begin{prooftree}
      \AxiomC{$\varphi, \Pi \vdash \psi$}
      \RightLabel{$\vdash \backslash$}
      \UnaryInfC{$\Pi \vdash \varphi \backslash \psi$}
    \end{prooftree}

    \begin{prooftree}
      \AxiomC{$\Pi, \varphi \vdash \psi$}
      \RightLabel{$\vdash /$}
      \UnaryInfC{$\Pi \vdash \varphi \: / \: \psi$}
    \end{prooftree}

    \begin{prooftree}
      \AxiomC{$\Gamma \vdash \varphi$}
      \AxiomC{$\Delta \vdash \psi$}
      \RightLabel{$\vdash \bullet$}
      \BinaryInfC{$\Gamma, \Delta \vdash \varphi \bullet \psi$}
    \end{prooftree}
\end{flushright}
\end{minipage}
\end{definition}

The class ${\bf R}(;, \setminus, /, \leq)$ consists of the following structures:

\begin{definition} \label{rrs}
  Let $A$ be a set of binary relations on some base set $W$ such that $R = \cup A$ is transitive and $W$ is a domain of $R$. A relation residuated semigroup is an algebra $\mathcal{A} = \langle A, ;, \setminus, /, \subseteq \rangle$ where for each $a, b \in A$
  \begin{enumerate}
    \item $a ; b = \{ (x, z) \: | \: \exists y \in W \: ((x, y) \in a \: \& \: (y, z) \in b) \}$,
    \item $a \setminus b = \{ (x, y) \: | \: \forall z \in W \: ((z, x) \in a \Rightarrow (z, y) \in b)\}$,
    \item $a \: / \: b = \{ (x, y) \: | \: \forall z \in W \: ((y, z) \in b \Rightarrow (x, z) \in a)\}$,
    \item $a \leq b$ iff $a \subseteq b$.
  \end{enumerate}
\end{definition}
A residuated semigroup is called \emph{representable} if it is isomorphic to some algebra belonging to ${\bf R}(;, \setminus, /, \leq)$.

\begin{definition}
  Let $\tau = \{ ;, \setminus, /, \leq \}$, $\mathcal{A}$ a $\tau$-structure, and $X$ a base set. An \emph{interpretation} $R$ over a base $X$ maps every $a \in \mathcal{A}$ to a binary relation $a^R \subseteq X \times X$. A \emph{representation} of $\mathcal{A}$ is an interpretation $R$ satisfying the following conditions:

  \begin{enumerate}
    \item $a \leq b$ iff $a^R \subseteq b^R$,
    \item $(a;b)^R = \{ (x, y) \: | \: \exists z \in X \: (x, z) \in a^R \: \& \: (z, x) \in b^R \} = a^R ; b^R$,
    \item $(a \setminus b)^{R} = \{ (x, y) \: | \: \forall z \in X \: ((z, x) \in a^R \Rightarrow (z, y) \in b^R)\} = a^R \setminus b^R$,
    \item $(a \: / \: b)^{R} = \{ (x, y) \: | \: \forall z \in X \: ((y, z) \in a^R \Rightarrow (x, z) \in b^R)\} = a^R \: / \: b^R$.
  \end{enumerate}
\end{definition}

Andr\'{e}ka and Mikul\'{a}s proved the representation theorem for ${\bf RS}$ (\cite{andreka1994lambek}) in the fashion of step-by-step representation. See this paper to learn more about step-by-step representations in general \cite{hirsch1997step}. That also implies relational completeness of the Lambek calculus, the logic of ${\bf RS}$.

This fact also claims that the theory of ${\bf R}(;, \setminus, /, \leq)$ is finitely axiomatisable since their theories coincide, and the class of all residuated semigroup is finitely axiomatisable.

One may rephrase the result of the theorem by Andr\'{e}ka and Mikul\'{a}s as $\mathcal{A}$ is representable iff $\mathcal{A}$ is a residuated semigroup, where $\mathcal{A}$ is a structure of the signature $\{ ;, \setminus, /, \leq \}$. Thus, it is sufficient to show that any finite residuated semigroup has a representation over a finite base in order to show that ${\bf R}(;, \setminus, /, \leq)$ has the finite representation property.

Let us start with preliminary order-theoretic definitions. Recall that a \emph{closure operator} on a poset $\langle P, \leq \rangle$ is a monotone map $j : P \to P$ satisfying $a \leq j a = j j a$ for each $a \in P$. Given $a \in P$, a \emph{lower cone} generated by $a$ is a subset $\downarrow a = \{ x \in P \: | \: x \leq a \}$.

A quantale is a complete lattice-ordered semigroup. Quantales has been initially introduced by Mulvey to provide a noncommutative generalisation of locales, study the spectra of $C^{*}$-algebras, and classify Penrose tilings, see \cite{mulvey1986suppl} \cite{mulvey2005noncommutative}.

\begin{definition}
  A \emph{quantale} is a structure $\mathcal{Q} = \langle Q, ;, \Sigma \rangle$ such that $\mathcal{Q} = \langle Q, \Sigma \rangle$ is a complete lattice, $\langle Q, ; \rangle$ is a semigroup, and the following conditions hold for each $a \in Q$ and $A \subseteq Q$:
  \begin{enumerate}
    \item $a \: ; \: \Sigma A = \Sigma \{ a ; q \: | \: q \in A \}$,
    \item $\Sigma A \: ; \: a = \Sigma \{ q ; a \: | \: q \in A \}$.
  \end{enumerate}
\end{definition}

Note that any quantale is a residuated semigroup as well. Given a quantale $\mathcal{Q} = \langle Q, ;, \Sigma \rangle$, One may express residuals uniquely with supremum and product as follows for each $a, b \in Q$:
\begin{enumerate}
  \item $a \setminus b = \Sigma \{ c \in Q \: | \: a ; c \leq b \}$,
  \item $a \: / \: b = \Sigma \{ c \in Q\: | \: b ; c \leq a \}$.
\end{enumerate}

A quantic nucleus is a closure operator on a quantale allowing one to define subalgebras. Such an operator is a generalisation of a well-known nucleus operator in locale theory and pointfree topology, see, e. g., \cite{bezhanishvili2016locales}. The following definition and the proposition below are due to \cite[Theorem 3.1.1]{rosenthal1990quantales}.
\begin{definition}
  A \emph{quantic nucleus} on a quantale $\langle A, ;, \Sigma \rangle$ is a mapping $j : A \to A$ such that $j$ a closure operator satisfying $j a ; j b \leq j (a ; b)$.
\end{definition}

\begin{proposition} \label{subsemi}
  Let $\mathcal{A} = \langle A, ;, \Sigma \rangle$ be a quantale and $j$ a quantic nucleus, then
  $\mathcal{A}_j = \{ a \in A \: | \: j a = a \}$ forms a quantale, where $a ;_j b = j(a ; b)$ and $\Sigma_j A = j (\Sigma A)$ for each $a, b \in {A}_j$ and $A \subseteq \mathcal{A}_j$.
\end{proposition}

One may embed any residuated semigroup into some quantale with the Dedekind-MacNeille completion (see, for example, \cite{theunissen2007macneille}) as follows. According to Goldblatt \cite{goldblatt2006kripke}, residuated semigroups have the following representation based on quantic nuclei and the Galois connection.

We need this construction to solve the problem, let us discuss it briefly. See Goldblatt's paper to have a complete argument \cite{goldblatt2006kripke}.

Let $\mathcal{A} = \langle A, \leq, ;, \setminus, / \rangle$ be a residuated semigroup. Then $\langle \mathcal{P}(A), ;, \bigcup \rangle$ is a quantale, where the product operation on subsets is defined with the pairwise products of their elements.

Let $X \subseteq A$. We put $lX$ and $uX$ as the sets of lower and upper bounds of $X$ in $A$. We also put $m X = lu X$.
Note that the lower cone of an arbitrary $x$ is $m$-closed, that is,
$m (\downarrow x) = \downarrow x$.

$m : \mathcal{P}(A) \to \mathcal{P}(A)$ is a closure operator and the set

\begin{center}
$(\mathcal{P}(A))_m = \{ X \in \mathcal{P}(S) \: | \: m X = X\}$ )
\end{center}
forms a complete lattice with $\Sigma_{m} \mathcal{X} = m ( \bigcup \mathcal{X})$ and $\Pi_{m} = \bigcap \mathcal{X}$ \cite{davey2002introduction}.

The key observation is that $m$ is a quantic nucleus on $\mathcal{P}(A)$, that is, $m A ; m B \subseteq m (A ; B)$. We refer here to the Goldblatt's paper mentioned above.

Thus, according to Proposition~\ref{subsemi}, $\langle (\mathcal{P}(A))_m, \subseteq, ;_m \rangle$ is a quantale itself since $m$ is a quantic nucleus.

Let us define a mapping $f_m : \mathcal{A} \to (\mathcal{P}(A))_m$ such that $f_m : a \mapsto \downarrow a$. This map is well-defined since any lower cone generated by a point is $m$-closed. Moreover, $f_m$ preserves products, residuals, and existing suprema. In particular, $f_m$ is a residuated semigroup embedding.

As a result, we have the following representation theorem.

\begin{theorem} \label{orsRep}
  Every residuated semigroup has an isomorphic embedding to the subalgebra of some quantale.
\end{theorem}

In their turn, quantales have a relational representation. First of all, let us define a relational quantale.
The notion of a relational quantale was introduced by Brown and Gurr to represent quantales as algebras of relations and study relational semantics of the full Lambek calculus, the logic of bounded residuated lattices, see, e.g., \cite{brown1993representation}.
\begin{definition}
  Let $A$ be a non-empty set. A \emph{relational quantale} on $A$ is an algebra $\langle R, \subseteq, ; \rangle$, where
  \begin{enumerate}
    \item $R \subseteq \mathcal{P}(A \times A)$,
    \item $\langle R, \subseteq \rangle$ is a complete join-semilattice,
    \item $;$ is a relational composition that respects all suprema in both coordinates.
  \end{enumerate}
\end{definition}

The uniqueness of residuals in any quantale implies the following quite obvious fact.
\begin{proposition}\label{ok}
  Let $\mathcal{A}$ be a relational quantale over a base set $X$, then for each $a, b \in \mathcal{A}$
  \begin{enumerate}
    \item $a \setminus b = \{ (x, y) \in X^2 \: | \: \forall z \in X ( (z, x) \in a \Rightarrow (z, y) \in b) \}$,
    \item $a \: / \: b = \{ (x, y) \in X^2 \: | \: \forall z \in X ( (y, z) \in b \Rightarrow (x, z) \in b )\}$.
  \end{enumerate}
\end{proposition}

Let us discuss the relational representation of quantales. This construction is due to Brown and Gurr \cite{brown1993representation}.

Let $\mathcal{Q}$ be a quantale and $\mathcal{G}(\mathcal{Q})$ a set of its generators. We define:

\begin{center}
  $\hat{a} = \{ \langle g, q \rangle \: | \: g \in \mathcal{G}(\mathcal{Q}), q \in \mathcal{Q}, g \leq a ; q \} \:\:\:\: \widehat{\mathcal{Q}} = \{ \hat{a} \: | \: a \in \mathcal{Q} \}$
\end{center}

The mapping $a \mapsto \hat{a}$ satisfies the following conditions:

\begin{enumerate}
\item $a \leq b$ iff $\hat{a} \subseteq \hat{b}$,

\item $\widehat{\Sigma A} = \Sigma \widehat{A}$, $\hat{a} ; \hat{b} = \widehat{a ; b}$, and $\langle \widehat{\mathcal{Q}}, \subseteq, \Sigma \rangle$ is a complete lattice,

\item $\langle \widehat{\mathcal{Q}}, \subseteq, ; \rangle$ is a relational quantale,

\item $\mathcal{Q}$ is isomorphic to $\langle \widehat{\mathcal{Q}}, \subseteq, ; \rangle$,

\item $a \mapsto \hat{a}$ is a quantale isomorphism.
\end{enumerate}

\begin{theorem} \label{quantaleRep}
  Every quantale $\mathcal{Q} = \langle Q, ;, \Sigma \rangle$ is isomorphic to a relational quantale on $Q$ as a base set.
\end{theorem}

We describe how we use Theorem~\ref{orsRep}, Proposition~\ref{ok}, and Theorem~\ref{quantaleRep} and constructions from their proofs to obtain an interpretation of the signature of residuated semigroups on ${\bf R}(;, \setminus, /, \leq)$.

Let $\mathcal{A}$ be a residuated semigroup and $\mathcal{Q}_{\mathcal{A}}$ is a quantale of Galois closed subsets of $\mathcal{A}$. $\widehat{\mathcal{Q}_{\mathcal{A}}}$ is the corresponding relational quantale. Let us define an interpretation $R : \mathcal{A} \to \widehat{\mathcal{Q}_{\mathcal{A}}}$ such that

\begin{center}
  $R : a \mapsto a^{R} = \widehat{\downarrow a}$
\end{center}

According to the lemma below, such an interpretation is a representation. As we have already said above, the mapping $a \mapsto \downarrow a$ is order-preserving. Moreover, this mapping commutes with products and residuals.

\begin{lemma}
  Let $\tau$ be a signature of residuated semigroups. An interpretation $R : \mathcal{A} \to \widehat{\mathcal{Q}_{\mathcal{A}}}$ such that $R : a \mapsto a^{R} = \widehat{\downarrow a}$ is a $\tau$-representation.
\end{lemma}

\begin{proof}
  By Theorem~\ref{orsRep}, $\mathcal{Q}_{\mathcal{A}}$ is isomorphic to $\widehat{\mathcal{Q}_{\mathcal{A}}}$. The isomorphism is established with the mapping $\downarrow{a} \mapsto \widehat{\downarrow a}$ according to Theorem~\ref{quantaleRep}.
  Residuals in $\widehat{\mathcal{Q}_{\mathcal{A}}}$ are well-defined by Proposition~\ref{ok}.
\end{proof}

Theorem~\ref{orsRep}, Theorem~\ref{quantaleRep}, and the lemma above imply the following statement.
\begin{corollary} \label{orsRep2}
  Every residuated semigroup is isomorphic to the subalgebra of some relational quantale.
\end{corollary}

In particular, the representation we proposed implies the solution to \cite[Problem 19.17]{hirsch2002relation}.
\begin{theorem} \label{solution}
  ${\bf R}(;, \setminus, /, \leq)$ has the finite representation property.
\end{theorem}

\begin{proof}
  $ $

  Let $\mathcal{A}$ be a finite residuated semigroup.

  The representation of $\mathcal{A}$ as a subalgebra of a relational quantale clearly belongs to ${\bf R}(;, \setminus, /, \leq)$. This representation has the form

\begin{center}
  $\widehat{\mathcal{A}} = \langle \{ \widehat{\downarrow a} \}_{a \in \mathcal{A}}, ;, \setminus, /, \subseteq \rangle$.
\end{center}

  Moreover, such a representation with the corresponding relational quantale has the finite base, if the original algebra is finite. The base set of the quantale $\widehat{\mathcal{Q}_{\mathcal{A}}}$ is the set of Galois stable subsets of $\mathcal{A}$, the cardinality of which is finite.
\end{proof}

The main corollary of Theorem~\ref{solution} is that the Lambek calculus has the finite model property. Thus, we have a semantical proof of decidability of the Lambek calculus. Before that, there were several algebraic proofs that the Lambek caclulus has the FMP \cite{buszkowski2008infinitary}, but the authors used to consider arbitrary algebras, not representable ones.
Alternatively, one may show that the Lambek calculus is decidable syntactically, that is, via cut elimination and the subformula property \cite{lambek1958mathematics}.

\begin{corollary} \label{fmp}
  The Lambek calculus is complete w.r.t finite relational models (has the FMP).
\end{corollary}

Moreover, finite axiomatisability and having the finite representation property of ${\bf R}(;, \setminus, /, \leq)$ imply that the membership of ${\bf R}(\tau)_{fin}$ is decidable.

\bibliographystyle{spmpsci}
\bibliography{Rogozin}

\begin{thebibliography}{10}
\providecommand{\url}[1]{{#1}}
\providecommand{\urlprefix}{URL }
\expandafter\ifx\csname urlstyle\endcsname\relax
  \providecommand{\doi}[1]{DOI~\discretionary{}{}{}#1}\else
  \providecommand{\doi}{DOI~\discretionary{}{}{}\begingroup
  \urlstyle{rm}\Url}\fi

\bibitem{andreka1999finite}
Andr{\'e}ka, H., Hodkinson, I., N{\'e}meti, I.: Finite algebras of relations
  are representable on finite sets.
\newblock The Journal of Symbolic Logic \textbf{64}(1), 243--267 (1999)

\bibitem{andreka1994lambek}
Andr{\'e}ka, H., Mikul{\'a}s, S.: Lambek calculus and its relational semantics:
  completeness and incompleteness.
\newblock Journal of Logic, Language and Information \textbf{3}(1), 1--37
  (1994)

\bibitem{andreka2011axiomatizability}
Andr{\'e}ka, H., Mikul{\'a}s, S.: Axiomatizability of positive algebras of
  binary relations.
\newblock Algebra universalis \textbf{66}(1-2), 7 (2011)

\bibitem{bezhanishvili2016locales}
Bezhanishvili, G., Holliday, W.H.: Locales, nuclei, and dragalin frames.
\newblock Advances in modal logic \textbf{11} (2016)

\bibitem{brown1993representation}
Brown, C., Gurr, D.: A representation theorem for quantales.
\newblock Journal of Pure and Applied Algebra \textbf{85}(1), 27--42 (1993)

\bibitem{buszkowski2008infinitary}
Buszkowski, W., Palka, E.: Infinitary action logic: Complexity, models and
  grammars.
\newblock Studia Logica \textbf{89}(1), 1--18 (2008)

\bibitem{davey2002introduction}
Davey, B.A., Priestley, H.A.: Introduction to lattices and order.
\newblock Cambridge university press (2002)

\bibitem{goldblatt2006kripke}
Goldblatt, R.: A kripke-joyal semantics for noncommutative logic in quantales.
\newblock Advances in modal logic \textbf{6}, 209--225 (2006)

\bibitem{hirsch2004finite}
Hirsch, R.: The finite representation property for reducts of relation algebra.
\newblock Manuscript, September  (2004)

\bibitem{hirsch2005class}
Hirsch, R.: The class of representable ordered monoids has a recursively
  enumerable, universal axiomatisation but it is not finitely axiomatisable.
\newblock Logic Journal of the IGPL \textbf{13}(2), 159--171 (2005)

\bibitem{hirsch1997step}
Hirsch, R., Hodkinson, I.: Step by step-building representations in algebraic
  logic.
\newblock Journal of Symbolic Logic pp. 225--279 (1997)

\bibitem{hirsch2001representability}
Hirsch, R., Hodkinson, I.: Representability is not decidable for finite
  relation algebras.
\newblock Transactions of the American Mathematical Society \textbf{353}(4),
  1403--1425 (2001)

\bibitem{hirsch2002relation}
Hirsch, R., Hodkinson, I.: Relation algebras by games.
\newblock Elsevier (2002)

\bibitem{hirsch2011positive}
Hirsch, R., Mikul{\'a}s, S.: Positive fragments of relevance logic and algebras
  of binary relations.
\newblock The Review of Symbolic Logic \textbf{4}(1), 81--105 (2011)

\bibitem{hirsch2013ordered}
Hirsch, R., Mikul{\'a}s, S.: Ordered domain algebras.
\newblock Journal of Applied Logic \textbf{11}(3), 266--271 (2013)

\bibitem{hirsch2020finite}
Hirsch, R., {\v{S}}emrl, J.: Finite representability of semigroups with demonic
  refinement.
\newblock arXiv preprint arXiv:2009.06970  (2020)

\bibitem{hodkinson2000axiomatizability}
Hodkinson, I., Mikul{\'a}s, S.: Axiomatizability of reducts of algebras of
  relations.
\newblock Algebra Universalis \textbf{43}(2-3), 127--156 (2000)

\bibitem{jipsen2002survey}
Jipsen, P., Tsinakis, C.: A survey of residuated lattices.
\newblock In: Ordered algebraic structures, pp. 19--56. Springer (2002)

\bibitem{jonsson1951boolean}
J{\"o}nsson, B., Tarski, A.: Boolean algebras with operators, i, ii.
\newblock American J. of Mathematics \textbf{73}, 891--939 (1951)

\bibitem{krull1968idealtheorie}
Krull, W.: Idealtheorie.
\newblock Springer (1968)

\bibitem{lambek1958mathematics}
Lambek, J.: The mathematics of sentence structure.
\newblock The American Mathematical Monthly \textbf{65}(3), 154--170 (1958)

\bibitem{lyndon1950representation}
Lyndon, R.C.: The representation of relational algebras.
\newblock Annals of mathematics pp. 707--729 (1950)

\bibitem{maddux2016finite}
Maddux, R.D.: The finite representation property fails for composition and
  intersection.
\newblock arXiv preprint arXiv:1604.01386  (2016)

\bibitem{mclean2016finite}
McLean, B., Mikul{\'a}s, S.: The finite representation property for
  composition, intersection, domain and range.
\newblock International Journal of Algebra and Computation \textbf{26}(06),
  1199--1215 (2016)

\bibitem{monk1964representable}
Monk, D., et~al.: On representable relation algebras.
\newblock The Michigan mathematical journal \textbf{11}(3), 207--210 (1964)

\bibitem{mulvey1986suppl}
Mulvey, C.J.: \&, suppl.
\newblock Rend. Circ. Mat. Palermo II \textbf{12}, 99--104 (1986)

\bibitem{mulvey2005noncommutative}
Mulvey, C.J., Resende, P.: A noncommutative theory of penrose tilings.
\newblock International Journal of Theoretical Physics \textbf{44}(6), 655--689
  (2005)

\bibitem{pentus1993lambek}
Pentus, M.: Lambek grammars are context free.
\newblock In: [1993] Proceedings Eighth Annual IEEE Symposium on Logic in
  Computer Science, pp. 429--433. IEEE (1993)

\bibitem{rosenthal1990quantales}
Rosenthal, K.I.: Quantales and their applications, vol. 234.
\newblock Longman Scientific and Technical (1990)

\bibitem{theunissen2007macneille}
Theunissen, M., Venema, Y.: Macneille completions of lattice expansions.
\newblock Algebra Universalis \textbf{57}(2), 143--193 (2007)

\bibitem{zaretskii1959representation}
Zaretskii, K.: The representation of ordered semigroups by binary relations.
\newblock Izvestiya Vysshikh Uchebnykh Zavedenii. Matematika (6), 48--50 (1959)

\end{thebibliography}

\end{document}